\documentclass[11pt]{article}

\date{}

\usepackage{amssymb}
\usepackage{amsthm}
\usepackage{epstopdf}
\usepackage{graphicx}
\usepackage{amsfonts}
\usepackage{hyperref}
\usepackage[latin2]{inputenc}
\usepackage{amsmath}
\usepackage{anysize}
\usepackage{url}
\usepackage{titlesec}
\usepackage{setspace}
\titleformat{\section}
  {\normalfont\large\bfseries}
  {\thesection}{1em}{}
\titleformat{\subsection}
  {\normalfont\bfseries}
  {\thesubsection}{1em}{}

\newtheorem{theorem}{Theorem}[section]
\newtheorem{lemma}[theorem]{Lemma}
\newtheorem{claim}[theorem]{Claim}
\newtheorem{corollary}[theorem]{Corollary}
\newtheorem{remark}[theorem]{Remark}

\newtheorem{proposition}[theorem]{Proposition}

\def\MA{Monge--Amp\`ere }
\def\K{K\"ahler }
\def\i{\sqrt{-1}}
\def\del{\partial}
\def\dbar{\bar\partial}
\def\ddbar{\del\dbar}
\def\ra{\rightarrow}

\newcommand{\RR}{\mathbb{R}}
\newcommand{\CC}{\mathbb{C}}
\def\del{\partial}

\def\o{\omega}

\def\vp{\varphi}
\def\beq{\begin{equation}}
\def\eeq{\end{equation}}
\def\bi#1{\bibitem{#1}}
\def\PSH{\mathrm{PSH}}
\def\SH{\mathrm{SH}}
\def\h#1{\hbox{#1}}
\def\Re{\hbox{\rm Re}}

\def\b#1{\bar{#1}}

\begin{document}
\title{Kiselman's principle, the Dirichlet problem for
the Monge--Amp\`ere equation, and rooftop obstacle problems}

\author{Tam\'as Darvas and Yanir A. Rubinstein}
\maketitle

\def\q{\quad}
\def\beq{\begin{equation}}
\def\eeq{\end{equation}}
\def\beqno{\begin{equation*}}
\def\eeqno{\end{equation*}}
\def\eaeq{\end{aligned}}
\def\baeq{\begin{aligned}}
\def\bpf{\begin{proof}}
\def\epf{\end{proof}}
\def\calA{{\mathcal{A}}}
\def\opcit{\underbar{\phantom{aaaaa}}}
\def\env{\operatorname{env}}
\def\bmin{b_{\env}}
\def\benv{b_{\env}}
\def\polishl{\char'40l}
\def\Kolodziej{Ko\polishl{}odziej} \def\Blocki{B\polishl{}ocki}

\begin{abstract}
First, we obtain a new formula for Bremermann type upper envelopes, that arise
frequently in convex analysis and pluripotential theory, in terms
of the Legendre transform of the convex- or plurisubharmonic-envelope of the boundary
data. This yields a new relation between solutions of the Dirichlet problem for the
homogeneous real and complex Monge--Amp\`ere equations and Kiselman's minimum principle.
More generally, it establishes partial regularity for a Bremermann envelope 
whether or not it solves the \MA equation.
Second, we prove the second order regularity of the solution of the free-boundary problem for the Laplace equation with a rooftop obstacle, based
on a new a priori estimate on the size of balls that lie above the non-contact set.
As an application, we prove that convex- and plurisubharmonic-envelopes of rooftop obstacles have bounded second derivatives.
\end{abstract}

\section{Introduction}

In this article we give a new formula for the solution
of the Dirichlet problem for the homogeneous real and complex \MA equation
(HRMA/HCMA) on the product of either a convex domain and Euclidean space in the real
case, or a tube domain and a K\"ahler manifold in the complex case.
This is partly inspired by Kiselman's minimum
principle \cite{Kiselman} and recent work of Ross--Witt-Nystr\"om \cite{RWN}.
Our formula involves the convex- or plurisubharmonic-envelope of a family of functions on the Euclidean
space or the manifold, and
the Legendre transform on the convex domain. Consequently, one could hope to develop
the existence and regularity theory
for both weak and strong solutions using such a formula. In this article and in its sequels we develop this approach.

The regularity properties of the Legendre transform
are classical. Thus, one is naturally led to study the regularity
properties of the convex- or plurisubharmonic-envelope of a family of functions.
In the case of single function with bounded second derivatives,
the regularity of such envelopes was studied by
Benoist--Hiriart-Urruty, Griewank--Rabier, and Kirchheim--Kristensen,
\cite{BenoistHiriartUrruty,GriewankRabier,KirchheimKristensen}
(see also \cite[\S X.1.5]{HL2}) in the convex case, and by
Berman and Berman--Demailly \cite{Be,BD}
in the plurisubharmonic (psh) setting.
The convex- or psh-envelope of a family of functions is, by definition, the corresponding envelope of
the (pointwise) infimum of that family. However,
already when the family consists of two functions, their minimum
is only Lipschitz. Thus, our second goal here is to extend the aforementioned regularity results to such a setting.

The approach we take to achieve this goal is to study,
more generally, the analogous {\it subharmonic envelope}.
The subharmonic
envelope of a `rooftop obstacle' of the form  $\min\{b_0,\ldots,b_k\}$
is, of course, just the solution of the free-boundary
problem for the Laplace equation associated to this obstacle.
Our first regularity result concerning envelopes is that
the solution to the free-boundary problem for the Laplace equation associated to a
such a rooftop obstacle, for functions $b_i$ with finite $C^2$ norm, also has finite
$C^2$ norm, along with an a priori estimate.
Aside from basic regularity tools from the
theory of free-boundary problems associated to the Laplacian,
this involves a new a priori estimate on the size of a ball
that lies between the rooftop and the envelope.
This result stands in contrast to the results of
Petrosyan--To \cite{PetrosyanTo} that show that
the subharmonic-envelope is $C^{1,\frac12}$ and no better
for more general rootop obstacles.

Since the subharmonic-envelope always lies above both the convex- and the
psh-envelope this allows us to establish the regularity of
the latter envelopes as well.

An important application
that makes an essential use of our results is the determination of the
Mabuchi metric completion of the space of K\"ahler potentials,
that is treated in a sequel \cite{D2014}.

\section{Main results}

Our first result concerns a new formula for the solution of the HRMA/HCMA on certain
product spaces. While the real result resembles the complex result, it is
not implied by it directly. Thus, we split the exposition into two
(\S\ref{HCMRSubSec}--\S\ref{HRMRSubSec}).
Our second result concerns the regularity of subharmonic-, convex-, and psh-envelopes
of a `rooftop' obstacle. The regularity of the latter two  (\S\ref{CVXPSHRegSubSec})
is a consequence of that of the former (\S\ref{FBPSubSec}).
In passing, we also establish the Lipschitz regularity of the psh-envelope
associated to a general Lipschitz obstacle.

\subsection{A formula for the solution of the HCMA}
\label{HCMRSubSec}

Suppose $(M,\omega)$ is a compact, closed and connected K\"ahler manifold and
let $K \subset \Bbb R^k$
be a bounded convex open set. Denote by $K^{\Bbb C} = K \times \Bbb R^k$
(considered as a subset of $\Bbb C^k$)
the convex tube with base $K$.
Let $\pi_2: K^\CC \times M \to M$ denote the natural projection,
and denote by
$\PSH(K^\CC\times M,\pi_2^\star\o)$ the set of 
$\pi_2^\star\o$-plurisubharmonic functions.
We seek bounded $\RR^k-$invariant solutions $\vp \in L^\infty \cap \PSH(K^\CC\times M,\pi_2^\star\o)$ of the problem
\beq\label{HCMAEq}
(\pi_2^\star\o+\i\ddbar \vp)^{n+k}=0 \h{\ in $K^\CC\times M$},
\q
\vp=v \h{\ on $\del K^\CC\times M$},
\eeq
where the boundary data $v$ is bounded, $\RR^k-$invariant and
$v_s:=v(s,\,\cdot\,)\in \PSH(M,{\omega}), s \in \del K$.

Some care is needed in defining the sense in which the boundary data is attained
since the functions involved are merely bounded.
In \eqref{HCMAEq}, by ``$\vp=v \h{\ on $\del K^\CC\times M$}$"
we mean
that for each $z\in M$ the convex function
$\vp_z:= \vp(\,\cdot\,,z)$ is continuous up to the boundary of $K$ and satisfies
$\vp_z|_{\partial K}=v_z$.
This choice of boundary condition implies that
$v_z \in C^0(\partial K)$,
and we will assume this condition on the boundary data throughout. 

The study of the Dirichlet problem for the complex \MA equation goes back to Bremermann and Bedford--Taylor \cite{Bremermann,BT}. In particular, their results show that one should look for the solution as an upper envelope:
\beq\label{uUpperEnvEq}
\vp:=\sup\{w\in L^\infty\cap\PSH(K^\CC\times M,\pi_2^\star\o): \ w \textup{ is }\RR^k\textup{-invariant and }
w|_{\partial K^\CC\times M}\le v\},
\eeq
generalizing the Perron method for the Laplace equation, where $w|_{\partial K^\CC\times M}\le v$ means that $\limsup_{s \to s_0} w(s,z) \leq v(s_0,z)$ for all $z \in M, \ s_0 \in \del K$.
It is not immediate, but as we will prove in Theorem \ref{envdual}, $\vp$ is upper semi-continuous on $K^\CC \times M$.
Assuming this for the moment, by Bedford--Taylor's theory
$\vp$ solves \eqref{HCMAEq}
(in general, 
further conditions are needed on $v$ in order to ensure
that $\vp|_{\del K^\CC\times M}=v$, as discussed below in Remark \ref{SolutionRemark}).

Our first result gives a different formula for expressing $\vp$, regardless of whether $\vp$ assumes $v$ on the boundary.
It involves the psh-envelope operator solely in the $M$ variables,
and the Legendre transform solely in the $K$ variables. The psh-envelope is
the complex analogue of the convexification operator (or double Legendre transform) in the real setting,
and is different than the upper envelope in that, roughly, it involves functions and not boundary
values thereof.
Given a family of upper semi-continuous bounded functions $\{f_a\}_{a\in\calA}$ parametrized
by a 
set $\calA$, set
$$
P\{f_a\}_{a\in\calA}:=\sup\{h\in\PSH(M,\o)\,:\, h(z)\le\inf_{a\in\calA}\{f_a(z)\}, \,\forall\, z\in M\}.
$$
As each $f_b$ is upper semi-continuous, it follows that the upper semi-continuous regularization satisfies $\textup{usc}(P\{f_a\}_{a\in\calA}) \leq f_b$, hence by Choquet's lemma $\textup{usc}(P\{f_a\}_{a\in\calA})$ is a competitor for the supremum, which in turn implies $P\{f_a\}_{a\in\calA}= \textup{usc}(P\{f_a\}_{a\in\calA})\in \PSH(M,\o)$.

Given a function $f=f(s,z)$ on $K\times M$ (that we consider
as a family of functions on $K$ parametrized by $M$), we let
\beq\label{LegTrans1Eq}
f^\star(\sigma,z)=f^\star(\sigma):=\inf_{s\in K}[f(s,z)-\langle\sigma,s\rangle].
\eeq
This is the {\it negative} of the usual Legendre transform solely in the $K$-variables, in particular, it maps convex functions to concave functions, and vice versa. Despite this, we also refer to it sometimes as the partial Legendre transform, and we often omit the dependence
of the function on the $M$ variables in the notation. Here $\langle\,\cdot\,,\,\cdot\,\rangle$
is the pairing between $\RR^k$ and its dual. Conversely,
if $g=g(\sigma,z)$ is a
function on $\RR^k\times M$ taking values in $[-\infty,\infty)$,
where $\RR^k$ is considered as the dual vector space to the copy of $\RR^k$
containing $K$, then
\beq\label{LegTrans2Eq}
g^\star(s,z)=g^\star(s):=\sup_{\sigma\in \RR^k}[g(\sigma,z)+\langle\sigma,s\rangle].
\eeq
Note that $f^{\star\star}=f$ if and only if $f$ is convex, lower semicontinuous and nowhere
equal to $-\infty$ (we do not allow the constant function $-\infty$), 
and otherwise $f^{\star\star}$ is the convexification of $f$, namely,
the largest convex function majorized by $f$
\cite{Mand,Fenchel,Rockafellar}.

\begin{theorem}
\label{envdual}
Assume that $v$ is bounded, $v_s=v(s,\,\cdot\,)\in \PSH(M,{\omega})$  and $v_z = v(\cdot,z) \in C^0(\partial K)$ for all $s\in\del K$, $z \in M$. Then $\vp$ as defined in \eqref{uUpperEnvEq} is upper semi-continuous and
\begin{equation}\label{main_identity}
\vp(h,z)=(P\{v_s-\langle s,\sigma\rangle\}_{s\in\del K})^\star(h,z)
=\sup_{\sigma\in\RR^k}[P\{v_s-\langle s,\sigma\rangle\}_{s\in\del K}(z)+\langle h,\sigma\rangle], \ h\in K,z \in M.
\end{equation}
Equivalently, $\vp^\star(\sigma,z)=
\inf_{s\in K}[\varphi(s,z)-\langle\sigma,s\rangle]=P\{v_s-\langle s,\sigma\rangle\}_{s\in\del K}(z)$.
\end{theorem}

To avoid confusion, we emphasize that
$P\{v_s-\langle s,\sigma\rangle\}_{s\in\del K}$ is {\it not}
the upper envelope of a family of linear function in $\sigma$ (that
would imply it is convex, which is essentially never true).
Instead, the psh-envelope of this family is a global operation
done for each $\sigma$ separately, and it is in fact concave in $\sigma$, as
the second statement in the theorem shows.

We pause to note an important corollary of this result for the special case $K=[0,1]$,
where $K^\CC$ is now the strip $S:=[0,1]\times\RR$, and 
\eqref{HCMAEq} becomes
\beq\label{MabuchiEq}
(\pi_2^\star\o+\i\ddbar \vp)^{n+1}=0, \q \vp|_{\{i\}\times\RR}=v_i,\; i=0,1,
\eeq

\begin{corollary}
\label{MabuchiCor}
Bedford--Taylor solutions of \eqref{MabuchiEq}
with bounded endpoints $v_0,v_1\in L^\infty(M)$, are given by
\begin{equation}\label{Mabuchimain_identity}
\vp(s,z)=P(v_0,v_1-\sigma)^\star_s(z)=\sup_{\sigma\in\RR}[P(v_0,v_1-\sigma)(z)+s\sigma], \ s \in [0,1], \ z\in M.
\end{equation}
\end{corollary}

According to Mabuchi, Semmes, and Donaldson \cite{M,S,D1}, sufficiently regular solutions of \eqref{MabuchiEq}
are geodesics in the Mabuchi metric on the space of \K potentials with respect to $\o$.
Thus, Corollary \eqref{MabuchiCor}
implies that Mabuchi's geometry is essentially determined
by the understanding of upper envelopes of the form
$P({v_0},v_1 - \tau)$, for all $v_0,v_1\in\PSH(M,\o)\cap L^\infty(M)$
and for all $\tau \in \Bbb R$. We refer to the sequel \cite{D2014} for
applications of Corollary \ref{MabuchiCor} in this direction, in particular, determining the metric completion of the Mabuchi metric.

\begin{figure}
\begin{center}
\includegraphics[trim=0cm 24cm 0cm 0cm,clip=true,width=4.6in]{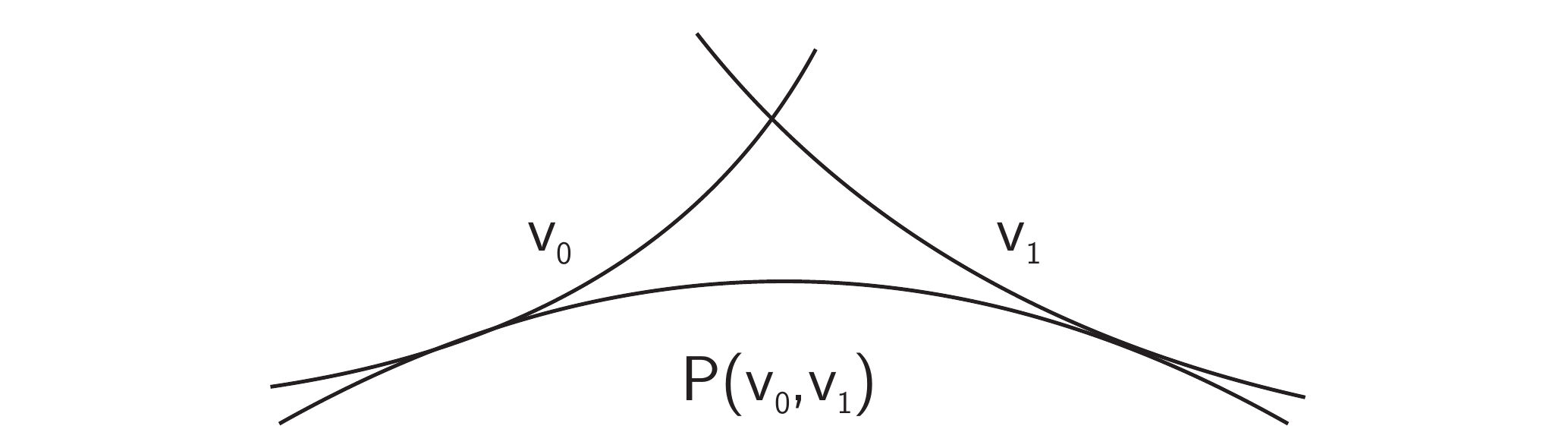}
\caption{The barriers $v_0,v_1$ and the envelope $P(v_0,v_1)$}
\label{figP(v0v1)}
\end{center}
\end{figure}

\subsection{A formula for the solution of the HRMA}
\label{HRMRSubSec}

Theorem \ref{envdual} has a convex analogue in the setting of the HRMA.
The result does not follow directly from the seemingly harder result
for the HCMA.
For concreteness, we only state the analogue of Corollary \ref{MabuchiCor}
in this setting, that arises in the setting of the Mabuchi metric on
a toric manifold $M$.
The reader is referred to \cite[\S2]{RZII} for the
relevant background concerning the HRMA and toric geometry.

For $z$ belonging to the open orbit of the complex torus $(\CC^n)^\star$
(that is dense in $M$), set $x=\Re\log z\in \RR^n$. On the open orbit,
$\o=\i\ddbar\psi_\o$ with $\psi_\o$ $(S^1)^n$-invariant, thus consider
$\psi_\o$ as a function on $\RR^n$.
Then, the HCMA \eqref{HCMAEq} reduces to the HRMA,
\beq\label{HRMAEq}
\h{MA}\psi(s,x)=0, \q \h{on } [0,1]\times\RR^n,
\qquad
\psi_i(x)\equiv\psi(i,x)=\psi_\o(x)+v_i(e^x), \q i\in\{0,1\}.
\eeq
Here, $\h{MA}$ is the unique continuous extension of the
operator $f\mapsto d\frac{\del f}{\del x_1}\wedge\cdots \wedge d\frac{\del f}{\del x_n}$
from $C^2(\RR^n)$ to the cone of convex functions on $\RR^n$.

The following is a convex version of Corollary \ref{MabuchiCor}.

\begin{proposition}\label{HRMAProp}
The solution of \eqref{HRMAEq} with convex endpoints $\psi_0,\psi_1$ is
given by
\begin{equation}\label{toricmain_identity}
\psi=(\min\{\psi_0,\psi_1-\sigma\}^{\star\star})^\star
=\sup_{\sigma\in\RR}[\min\{\psi_0,\psi_1-\sigma\}^{\star\star}+s\sigma].
\end{equation}
\end{proposition}

Here the first (innermost) two Legendre transforms are in the $x$ variables,
while the third (outermost) negative Legendre transform is in the $\sigma$
variable.  
Note that, strictly speaking, this result is not a consequence
of Corollary \ref{MabuchiCor}, since it involves the {\it potentially larger}
convex envelope (the supremum is taken over convex functions that might not
come from toric potentials) and not the psh-envelope;
 rather, 
 Proposition \ref{HRMAProp}
implies Corollary \ref{MabuchiCor} (in this symmetric setting)
since it shows
that the psh-envelope in this setting is attained at a `toric' convex function.

This formula also has an interpretation in terms of Hamilton--Jacobi equations, in the spirit of \cite{RZ}, that we discuss elsewhere.

\subsection{Regularity of rooftop subharmonic-envelopes}

\label{FBPSubSec}

The following  result plays a crucial role in our
proof of the regularity of convex- and psh-envelopes of rooftop obstacles.
It is of independent interest to the study of regularity of solutions
to the free-boundary problem for rooftop obstacles for the Laplacian.
The solution of the aforementioned free-boundary value problem
is, in fact, the subharmonic-envelope of rooftop obstacles. This is a purely local
result, and is stated on the open unit ball $B_1$ in $\RR^n$
(we let $B_R(x_0)$ denote the ball of radius $R$ centered at $x_0\in\RR^n$;
when $x_0=0$ we write $B_R=B_R(0)$). Denote by $\SH(B_1)$
the set of subharmonic functions on $B_1$.

\begin{theorem} \label{P_reg_prop}
Let $b_0,b_1 \in C^{1,1}(B_1)$, and
let
\beq\label{bminEq}
\bmin:=\sup\{ f\in\SH(B_1)\,:\, f\le\min\{b_0,b_1\}\}.
\eeq
Then, there exists a constant $C=C(n,\| b_0\|_{C^{2}(B_1)},\| b_1\|_{C^{2}(B_1)})$ such that
$$
\| \bmin\|_{C^{2}(B_{1/8})} \leq C.
$$
\end{theorem}

\subsection{Regularity of the convex-envelope or psh-envelope of a family of functions}
\label{CVXPSHRegSubSec}

Given an upper semi-continuous family  $\{ f_a\}_{a \in \mathcal A}$
with additional regularity properties, one would like to study how
much regularity is preserved by the envelope 
$P\{ f_a\}_{a \in \mathcal A}$. 
%
%
Motivated by Corollary \ref{MabuchiCor} and Proposition \ref{HRMAProp}, we are led to study
the regularity of upper envelopes of the type $P(v_0,v_1)$.
Here, we concentrate on the case when the
{\it barriers} (sometimes also called {\it obstacles})
$v_0$ and $v_1$ are rather regular.
The sequel \cite{D2014} treats the case when $v_0$ or $v_1$
are rather irregular in the psh setting. Already in the case of smooth convex functions,
the convexification is not $C^2$ in general. Thus, the following
results gives conditions that guarantee essentially optimal regularity.
A novelty of our approach, perhaps, is that both the convex- and the psh-envelopes are handled simultaneously.

To state the results, we define the Banach space
\beq\label{CooEq}
C^{1\bar1}(M):=\{f\in L^\infty(M)\,:\, \Delta_\o f\in L^\infty(M)\},
\eeq
with associated Banach norm
\beq\label{CooNormEq}
||f||_{C^{1\bar1}}:=||f||_{L^\infty(M)}+||\Delta_\o f||_{L^\infty(M)}.
\eeq
If $Cf\in\PSH(M,\o)$ for some $C>0$, then $f\in C^{1\bar1}(M)$ if and only if $\i\ddbar f$
is a current with bounded coefficients.
We also define, as usual, $C^{1,1}(M)$ to be the
Banach space of functions on $M$ with finite $C^2(M)$ norm.
One has $C^2(M)\subset C^{1,1}(M)\subset C^{1\bar 1}(M)$.

\begin{theorem}
\label{PRegThm}
One has the following estimates:\hfill\break
(i)
$\|P(v)\|_{C^{1}}\le C(M,\o,\| v\|_{C^{1}})$.
\hfill\break
(ii) 
$\|P(v_0,v_1)\|_{C^{1\bar1}}
\leq C(M,\o,\| v_0\|_{C^{1\bar1}}
,\| v_1\|_{C^{1\bar1}}).
$
\hfill\break
(iii)
Suppose $[\omega_0] \in H^2(M,\Bbb Z)$. Then,
$\|P(v_0,v_1)\|_{C^{2}} \leq C(M,\o,\| v_0\|_{C^{2}},\| v_1\|_{C^{2}}).$
\end{theorem}

Our convention here and below is that the constants $C$ on the right
hand side of the estimates just stated may equal to $\infty$
only if the corresponding norms of $v$ or $v_i$ are infinite.

An analogous result can be stated for convex rooftop envelopes. For simplicity, we only state
a representative result in the toric setting of Proposition \ref{HRMAProp}.

\begin{corollary}
\label{PRealRegThm}
Let $\psi_0,\psi_1$ 
be as in Proposition \ref{HRMAProp}. Then,
$$
\|\min\{\psi_0,\psi_1\}^{\star\star}\|_{C^{2}} \leq C(M,\o,\| \psi_0\|_{C^{2}},\| \psi_1\|_{C^{2}}).
$$
\end{corollary}

By repeated application of the formula $P(v_0,v_1,\ldots,v_k)=P(v_0,P(v_1,\ldots,v_k))$,
the results just stated hold also for
envelopes of the type $P(v_0,\ldots,v_k)$.

In general, the convex- or subharonic-envelope of a Lipschitz function
will be no better than Lipschitz, as shown by Kirchheim--Kristensen \cite{KirchheimKristensen}, 
and by Caffarelli \cite[Theorem 2]{Caf}, respectively.
Theorem \ref{PRegThm} (i) is the analogous fact for psh-envelopes.
The psh-envelope
of a family of functions, e.g., $P(v_0,v_1)$ is
of course the psh-envelope of the single function $\min\{v_0,v_1\}$
that is in general only Lipschitz. Thus, the point of Theorem
\ref{PRegThm} (ii)--(iii) is that for special Lipschitz functions of
the form $\min\{v_0,\ldots,v_k\}$ that we refer to as {\it
rooftop functions} (see Figure \ref{figP(v0v1)}) the psh-envelope
has a regularizing effect, roughly gaining a derivative.

The proof of Theorem \ref{PRegThm} uses 
basic techniques from the theory of free boundary problems for the Laplacian,
together with results of Berman \cite{Be} and Berman-Demailly \cite{BD}
 on upper envelopes of psh functions. Part (i) is, in fact, a simple
 consequence of the Lipschitz estimate of \Blocki\ \cite{BlockiGrad} 
in conjunction with the ``zero temprature" approximation procedure of Berman
\cite{BeApprox}. The bulk of the proof is thus devoted to parts (ii)--(iii). 
The key step is to show that there exists a $C^{1,1}$ function $b$
(a `barrier')
{\it along with an a priori estimate} depending only on the respective
norms of the $v_i$, such that $b$ lies below $\min\{v_0,v_1\}$ but
above $P(v_0,v_1)$. The barrier we construct is actually obtained
by first constructing local {\it subharmonic-envelopes}
of $v_0$ and $v_1$ on coordinate charts. This construction
is mostly based on well-known techniques from the study of the
free boundary Laplace equation, see, e.g.,
\cite{Caf,CafSalsa,PetrosyanSU},
but with one essential new ingredient, that we now describe.
For a general rooftop obstacle (that is, not necessarily of the form
$\min\{v_0,v_1\}$) Petrosyan--To \cite{PetrosyanTo} show that
the subharmonic-envelope is $C^{1,\frac12}$ 
and no better.
Yet, also in the literature on subharmonic-envelopes we were not able
to find the regularization statement for rooftop obstacles
of the form  $\min\{v_0,v_1\}$
although it might very well be known to experts.
Thus,
the main new technical ingredient is the estimate of Proposition
\ref{cushintheorem}
that guarantees that around each point in the set $\{v_0=v_1\}$
there exists a ball of a priori estimable size that stays
away from the {\it contact set}, i.e., the set where
the local subharmonic envelope equals the barrier $\min\{v_0,v_1\}$.
Given this estimate, the standard quadratic growth estimate
carries over to our setting, and one obtains a priori estimates on $b$.
Then, since the subharmonic-envelope necessarily majorizes the psh-envelope,
we get $P(b)=P(v_0,v_1)$, to which one may apply Berman--Demailly's results.

A regularity result of a similar nature has been recently proved by Ross--Witt-Nystrom \cite{RWN1}
in a different setting.
Namely, they study regularity of envelopes of the type $P_{[\phi]}(\psi)= \textup{usc} \big( \sup_{c >0} P(\phi+c,\psi)\big)$, 
where $\psi \in C^{1\bar 1}(M)$, $\phi \in \PSH(M,\o)$ is exponentially H\"older continuous and $M$ is polarized.
Also, upon completing this article, we were informed by Berman that the
technique of \cite{BD} can be extended to prove Theorem \ref{PRegThm}(iii) \cite{Berman-personal}.
Perhaps the novel point in our approach, compared to such an extension, is
that it also gives, in passing, a useful result concerning the obstacle problem
for the Laplacian, and thus proves the regularity of the subharmonic-, convex-, and psh-envelopes, all at once.

\subsection{Applications to regularity of Bremermann upper envelopes}

A combination of Theorem \ref{envdual} and Theorem \ref{PRegThm} (i) gives
fiberwise Lipschitz regularity of the Bremermann upper envelope $\vp$ \eqref{uUpperEnvEq}
associated to fiberwise Lipschitz boundary data. This provides an instance when one can draw conclusions about the regularity of $\vp$ by studying first the regularity of its partial Legendre transform.

\begin{corollary}\label{LipschitzHCMACor} In the setting of Theorem \ref{envdual}, 
the envelope $\vp$ satisfies
$$
\|\varphi(s,\cdot) \|_{C^1} \leq C(M,\o,\sup_{s \in \partial K}\|v(s,\cdot)\|_{C^1}),
\q \h{for any $s \in K$}.
$$ 
In other words, if the boundary data is fiberwise Lipschitz, so is the envelope,
and with a uniform estimate.
\end{corollary}

The novelty of this result is that it proves regularity of the envelope $\vp$,
whether or not it solves the HCMA. We are not aware of any such results in the literature.
At the same time, when $\vp$ does solve the HCMA then other techniques exist,
notably \Blocki's Lipschitz estimate \cite{BlockiGrad}. However, even then our method seems to be new
in that it furnishes fiberwise Lipschitz regularity given the same on the boundary data,
while \Blocki's estimate alone gives full Lipschitz regularity starting from 
full (also in the $\del K$ directions) Lipschitz regular data. Of course, it should
be stressed that we ultimately use \Blocki's estimate in our proof, but we do so
only in the fiberwise directions.

\subsection*{Organization}

Theorem \ref{envdual} and Corollary \ref{MabuchiCor} are proved in \S\ref{TubeSec}.
The convex analogue, Proposition \ref{HRMAProp}, is proved in \S\ref{CVXSubSec}.
Theorem \ref{PRegThm} (i) concerning Lipschitz regularity of the psh-envelope is proved in \S\ref{LipRegPSHSubSec},
where we also prove Corollary \ref{LipschitzHCMACor}.
Theorem \ref{PRegThm} (ii)--(iii) and Corollary \ref{PRealRegThm}, concerning the
regularity of second derivatives of the psh- and convex-envelopes, are proved in \S\ref{RegRoofCVXPSHSubSec}.
Finally, the main regularity result concerning the subharmonic envelope, Theorem \ref{P_reg_prop},
is proved in \S\ref{proofofP_reg_prop}.

\section{The Dirichlet problem on the product of a tube domain and a manifold}
\label{TubeSec}

Suppose that $f(s,z)$ is a convex function on $\RR^k_s\times\RR^m_z$. Then
$\inf_sf(s,z)$ is either identically $-\infty$, or else a convex function on $\RR^m$
\cite[Theorem 5.7; p. 144]{Rockafellar},\cite[Theorem 1.3.1]{Kiselman-notes}.
If we replace ``convex" with ``psh" and $\RR$ by $\CC$ this is not true in general.
A special situation in which this is true was described by Kiselman. Let us recall a local version of this result \cite{Kiselman}
(cf. \cite[Theorem I.7.5]{De}).
As in \S\ref{HCMRSubSec}, let
$K\subset\RR^k$ be a convex set and denote by
$K^\CC:=K+\i\RR^k\subset \CC^{k}$ the tube domain associated to $K$.
Denote by $s$ a coordinate on $K\subset \RR^k$ and by $\tau:=s+\i t$ a
coordinate on $K^\CC\subset \CC^k$.

\begin{theorem}\label{KiselmanThm} Let $D \subset \Bbb C^n$ be a domain.
If $v\in\PSH(K^\CC\times D)$ is such that
that $v(s+\i t,z)=v(s,z)$ for all $t\in\RR^k$ then
\begin{equation}\label{Kiselman}
v(z)= \inf_{\tau \in K^\CC}v(\tau,z) 
\end{equation}
is either identically $-\infty$, or else psh on $D$.
\end{theorem}

\def\o{\omega}

This immediately implies the following global version.
As in \S\ref{HCMRSubSec},  we denote by $(M,\o)$ a K\"ahler manifold
and by $\pi_2:K^\CC\times M\ra M$, $\pi_1:K^\CC\times M\ra K^\CC$ the natural projections.

\begin{corollary}
\label{KiselmanCor}
Assume that $f\in \PSH(K^\CC\times M,\pi_2^\star\o)$ satisfies
$f(s,z)=f(s+\i t,z)$ for all $t\in\RR^k$.
Then $f^\star(\sigma,z)$, as defined in \eqref{LegTrans1Eq},
satisfies $f^\star(\sigma,\,\cdot\,)\in\PSH(M,\o)$ for each
$\sigma\in \RR^k$.
\end{corollary}

\begin{proof}[Proof of Theorem \ref{envdual}] 
We argue that $\vp$ is upper semi-continuous parallel with the proof of the formula
\begin{equation}\label{main_identity_proof}
\vp^\star(\sigma,z)=
\inf_{s\in K}[\varphi(s,z)-\langle\sigma,s\rangle]=P\{v_s-\langle s,\sigma\rangle\}_{s\in\del K}(z), \ \sigma \in \Bbb R^k, \ z \in M.
\end{equation}

To start, observe that both
$\vp(\,\cdot\,,z)$ and $(\textup{usc}\,\vp)(\,\cdot\,,z)$ are convex and bounded functions on $K$ for each $z\in M$ (note that $\sup\textup{usc}\,\vp= \sup\vp$). Indeed, the former is a supremum of convex functions , where as the latter is the restriction to $K \times \{z\}$ of an $\RR^k$-invariant $\omega-$psh function by Choquet's lemma.
Thus, it suffices to prove that
\beq
\label{LegKeyEq}
\vp^\star(\sigma,z) = (\textup{usc}\,\vp)^\star(\sigma,z),
\eeq
for all $\sigma \in \RR^k$ since then,
by applying another partial Legendre transform it follows that
$\vp=\textup{usc}\,\vp$.  The proof of \eqref{LegKeyEq} will be implicit in the proof of \eqref{main_identity_proof} below.

Recall that by Bedford-Taylor theory \cite[Theorem 1.22]{Kol} the set
$E = \{ \vp  < \textup{usc}\,\vp\} \subset K^\CC \times M$ has capacity zero, in particular its Lebesgue measure is also zero (meaning that $\int_E dV_{K^\CC \times M}=0$ for any smooth volume form $dV_{K^\CC \times M}$ on $K^\CC \times M$). As both $u$ and $\textup{usc}\,u$ are $\RR^k$-invariant, $E$ is
also $\RR^k$-invariant with base $B \subset K \times M$ (note that $B$ is \emph{not} a subset of $K$). Clearly, the Lebesgue measure
of $B$ is zero. For $z \in M$ we introduce the sets
$$
B_z = \pi_1(B \cap K \times \{ z\}) \subset K.
$$
It follows that $B_z$ has Lebesgue measure zero for all $z \in M \setminus F$, where $F \subset M$ has Lebesgue measure zero.

Suppose $z  \in M \setminus F$, we claim that in fact $B_z$ is empty. This follows, as the continuous convex functions $\vp(\,\cdot\,,z)$ and
$(\textup{usc}\,\vp)(\,\cdot\,,z)$ agree on the dense set $K \setminus B_z$, hence they have to agree on all of $K$, hence $B_z$ is empty. This implies that
\begin{equation}\label{uscIndetity}
\vp^\star(\sigma,z) = (\textup{usc}\,\vp)^\star(\sigma,z),
\quad \textup{ for all } z \in M \setminus F,  \sigma \in \RR^k.
\end{equation}
Now, by Corollary \ref{KiselmanCor}, for each $\sigma\in \RR^k$,
the function $(\textup{usc } \vp)^\star(\sigma,\,\cdot\,)$ belongs to $\PSH(M,\o)$.
Moreover, by definition of $\vp$ we have
\begin{equation}\label{vpest}
\vp^\star_\sigma \le \vp_s-\langle\sigma,s\rangle \textup{ for all }s\in K.
\end{equation}
We can in fact extend this estimate to the boundary of $\del K$:
\begin{claim}
For all $s\in\del K$ and $\sigma\in\RR^k$,
 $\vp^\star(\sigma,z) \le v_s(z)-\langle\sigma,s\rangle$.
\end{claim}

Indeed, as $v_z \in C(\del K)$ for all $z \in M$, it follows that $\vp(p,z) \leq \mathcal P[v_z](p), \ p \in K$, where $\mathcal P[v_z] \in C(\overline K)$ is the harmonic function on $K$ satisfying $\mathcal P[v_z]|_{\partial K}=v_z.$ This implies that $\limsup_{p \to s}\vp_p(z) \leq v_s(z)$ for all $z \in M, \ s \in \partial K$, hence we can take the $\limsup$ of the right hand side of \eqref{vpest} to conclude the claim.

Thus, by \eqref{uscIndetity} we also have $(\textup{usc}\,\vp)^\star(\sigma,z)
\le v_s(z)-\langle\sigma,s\rangle$ for $z \in M \setminus F, \ s \in \partial K$. As $F$ has Lebesgue measure zero we claim that this inequality extends to all $z \in M$. This follows from the fact that  $(\textup{usc}\,\vp)^\star_\sigma$ and $ v_s-\langle\sigma,s\rangle$ are $\omega-$psh for fixed $s \in \partial K$, hence by the sub-meanvalue property we can write:
$$(\textup{usc}\,\vp)^\star_\sigma(z) = \lim_{r \to 0} \oint_{B(z,r)}(\textup{usc}\,\vp)^\star_\sigma(\xi)dV(\xi) \leq \lim_{r \to 0} \oint_{B(z,r)}(v_s(\xi)-\langle\sigma,s\rangle) dV(\xi) =\vp_s(z)-\langle\sigma,s\rangle,$$
for all $z \in M$, where $B(z,r)$ is a coordinate ball around $z$ and $dV$ is the standard Euclidean measure in local coordinates.

Thus,
$(\textup{usc}\,\vp)^\star(\sigma,\cdot)
$ is a competitor in
the definition of
$P\{v_s-\langle s,\sigma\rangle\}_{s\in\del K}$ concluding that
\begin{equation}\label{LegendreIneq1}
\vp^\star(\sigma,\,\cdot\,)
\le
(\h{usc}\,\vp)^\star(\sigma,\,\cdot\,)
\le P\{v_s-\langle\sigma,s\rangle\}_{s \in \partial K}.
\end{equation}

Conversely, let $\chi\in\PSH(M,\o)$ satisfy $\chi\le
v_a-\langle a,\sigma\rangle$ for each
${a\in\del K}$.
We claim that $\chi \le \varphi_s-\langle s,\sigma\rangle$
 for every $s\in K$. Indeed, by \eqref{uUpperEnvEq},
$$
\vp_s-\langle s,\sigma\rangle=\sup\{w_s-\langle s,\sigma\rangle\in L^\infty\cap\PSH(K^\CC\times M,\pi_2^\star\o):
(w-\langle s,\sigma\rangle)|_{\partial K^\CC}\le v-\langle s,\sigma\rangle\},
$$
so $\chi$ is a competitor in this last supremum, proving the claim.
Now, taking the infimum over all $s\in K$ it follows that
\begin{equation}\label{LegendreIneq2}
\vp^\star(\sigma,\cdot) \ge P\{v_s-\langle s,\sigma\rangle\}_{s\in\del K}.
\end{equation}
Putting together \eqref{LegendreIneq1} and \eqref{LegendreIneq2} the identities \eqref{main_identity_proof} and \eqref{LegKeyEq} follow, proving that $u$ is upper semi-continuous.
\end{proof}

\begin{remark}\label{SolutionRemark} 
{\rm
To guarantee that $\vp$ defined by \eqref{uUpperEnvEq} is an actual solution of \eqref{HCMAEq}, 
one can, e.g., assume that there exists a subsolution, by which we mean
an $\RR^k$-invariant $w\in L^\infty\cap\PSH(K^\CC\times M,\pi_2^\star\o)$ satisfying $w|_{\partial K^\CC}=v$. 
In fact, if such a subsolution exists, then $w_z \leq \vp_z$, implying that 
$\vp_z|_{\del K}$ lies above the boundary data. On the other hand, 
$\vp_z \leq \mathcal P[v_z]$,  
where $\mathcal P[v_z] \in C(\overline K)$ is the harmonic function on $K$ satisfying 
$\mathcal P[v_z]|_{\partial K}=v_z.$ Thus, $\vp_z|_{\del K}$ also lies below the boundary data.
In sum, $\vp|_{\del K^\CC}=v$.
}
\end{remark}

Providing a subsolution is often possible given special properties of $K$ or the boundary data $v$. An instance of this is the situation described in Corollary \ref{MabuchiCor}:
\begin{proof}[Proof of Corollary \ref{MabuchiCor}] 
By Theorem \ref{envdual},
all one needs to  verify is that $\varphi$, as defined in \eqref{uUpperEnvEq}, satisfies $\vp|_{\{i\}\times\RR}=v_i,\; i=0,1$. Formula \eqref{Mabuchimain_identity} follows then from \eqref{main_identity}. However, by an observation of Berndtsson
\cite{Br} we have that the function $w(s,z) = \max\{ v_0(z) - As,v_1(z) + A(1-s)\} \in PSH(K^\CC \times M,\pi_2^*\o)$ satisfies  $\psi|_{\{i\}\times\RR}=v_i,\; i=0,1,$ where $A = \max \{ \| v_0\|_{L^\infty},\| v_1\|_{L^\infty}\}.$ 
Hence, $w$ is a subsolution in the sense of Remark \ref{SolutionRemark}.
\end{proof}

We remark in passing that the general argument to prove upper semicontinuity
given in Theorem \ref{envdual} can be avoided in the special setting of Corollary \ref{MabuchiCor}
(i.e., when $K=[0,1]$). Indeed, by convexity in $s$,
$\vp(s,z) \leq s v_0(z) + (1-s)v_1(z)$ for all $(s,z)\in[0,1]\times M$,
thus also $\textup{usc}\,\vp$ satisfies the same inequality.
This last estimate in turn implies that $\textup{usc}\,\vp$ is a candidate in the supremum defining 
$\vp$, thus $\textup{usc}\,\vp = \vp$ (cf. \cite{Br}).

\subsection{A convex version for the HRMA}
\label{CVXSubSec}

In this subsection we prove the a version of Corollary \ref{MabuchiCor}
for the homogeneous real \MA (HRMA) equation.
While a proof of Proposition \ref{HRMAProp} and even its generalization to
higher dimensional $K$ can be given along very similar lines to the proof
of Theorem \ref{envdual}, we give below a somewhat different argument.

\begin{proof}[Proof of Proposition \ref{HRMAProp}]
As observed by Semmes \cite{Semmes1988,S}, the HRMA is linearized
by the partial Legendre transform {\it in the $\RR^n$ variables}.
Thus, the solution to the HRMA is given by
\beq\label{DoubleLegEq}
\psi(s,x)=((1-s)\psi_0^\star+s\psi_1^\star)^\star(x),
\eeq
where
$\psi_i^\star(y)=\sup_{y\in\RR^n}[\langle x,y\rangle-\psi_i(y)]$.
As is well-known, this is equal to the {\it infimal convolution
of $\psi_0$ and $\psi_1$} \cite[Theorem 38.2]{Rockafellar},
\beq\label{InfConvEq}
\inf_{\{x_0,x_1\in\RR^n\,:\, (1-s)x_0+sx_1=x\}}[(1-s)\psi_0(x_0)+s\psi_1(x_1)].
\eeq
This also follows directly from the fact that $\psi$ solves the HRMA,
since by \cite{RZ} a solution of the HRMA solves a Hamilton--Jacobi equation,
and \eqref{InfConvEq} is just the Hopf--Lax formula in that setting.
Now, we take the negative Legendre transform of \eqref{DoubleLegEq} in $s$
to obtain,
$$
\begin{aligned}
\psi^\star_x(\sigma)
&=\min_{s\in[0,1]}
\Big[
\inf_{\{x_0,x_1\in\RR^n\,:\, (1-s)x_0+sx_1=x\}}[(1-s)\psi_0(x_0)+s\psi_1(x_1)]
-s\sigma
\Big]
\cr
&
=
\min_{s\in[0,1]}\Big[
\inf_{\{x_0,x_1\in\RR^n\,:\, (1-s)x_0+sx_1=x\}}[(1-s)\psi_0(x_0)+s(\psi_1(x_1)-\sigma)]
\Big].
\end{aligned}
$$
Now we will show that this last expression is equal to
\beq\label{CvxHullEq}
\sup\{ v\,:\, v \h{\ is convex on } \RR^n \h{\ and $v\le\min\{\psi_0,\psi_1-\sigma\}$}\}=\min\{\psi_0,\psi_1-\sigma\}^{\star\star}.
\eeq
Fix $x\in\RR^n$, and
let $s\in[0,1]$ and $x_0,x_1\in\RR^n$ be
such that $(1-s)x_0+sx_1=x$. Let $v$ be a convex
function satisfying $v\le \min\{\psi_0,\psi_1-\sigma\}$.
Then,
$$
(1-s)\psi_0(x_0)+s(\psi_1(x_1)-\sigma)
\ge
(1-s)v(x_0)+sv(x_1)
\ge v(x),
$$
by
convexity of $v$.
Thus, $\psi^\star_x(\sigma)\ge \min\{\psi_0,\psi_1-\sigma\}^{\star\star}$.

Conversely,  the expression \eqref{InfConvEq}
is a convex function jointly in $s$ and $x$ (since it is evidently convex
in $x$ by \eqref{DoubleLegEq} and it solves the HRMA in all variables). By
the minimum principle for convex functions 
then $\psi^\star_x(\sigma)$ is convex in $x$.
By the definition of the negative Legendre transform in $s$,
$\psi^\star_x(\sigma)\le \min_{s\in\{0,1\}}[\psi_s(x)-s\sigma]
=\min\{\psi_0(x),\psi_1-\sigma\}$. Thus, $\psi^\star_x(\sigma)$
is a competitor in the left hand side of \eqref{CvxHullEq}.
Hence, $\psi^\star_x(\sigma)\le \min\{\psi_0,\psi_1-\sigma\}^{\star\star}$.
\end{proof}

\section{Regularity of upper envelopes of families
}

The bulk of this section is devoted to the proof of Theorem  \ref{PRegThm} (ii)--(iii)
and Corollary \ref{PRealRegThm}
that establish the regularity of psh- and convex-envelopes envlopes
associated to obstacles of the form $\min\{b_0,b_1\}$, 
that we refer to as `rooftop' envelopes (see Figure
\ref{figP(v0v1)}). However, we begin by first proving the Lipschitz regularity
of psh-envelopes (Theorem \ref{PRegThm} (i)).

\subsection{Lipschitz regularity of psh-envelopes}
\label{LipRegPSHSubSec}
Let $v \in C^\infty(M)$. 
Berman developed the following approach for constructing $P(v)$,
generalizing a related construction for obtaining ``short-time" solutions
to the Ricci continuity method, introduced in \cite{R08}, in turn based on a result of Wu \cite{Wu}
(a new approach to which has been given in \cite[\S9]{JMR}, see \cite[\S6.3]{R14}
for an exposition of these matters). 
For $\beta$ positive and sufficiently large one considers the equations
\begin{equation}\label{approx_equation}
(\o + \i\del\dbar u_\beta)^n = e^{\beta (u_\beta - v)}\o^n.
\end{equation}
By the classical work of Aubin and Yau, \eqref{approx_equation} admits a smooth solution
$u_\beta$. 
Berman proves that, as $\beta$ tends to infinity, $u_\beta$ converges to $P(v)$ uniformly, and that, moreover,
there is an a priori Laplacian estimate in this setting \cite{BeApprox}.
In this section we observe that, as expected, also an a priori Lipschitz estimate holds,
by directly applying \Blocki's estimate. In other words, we prove Theorem
\ref{PRegThm} (i). The proof will show that the 
constant in Theorem \ref{PRegThm} (i) depends on a lower bound
of the bisectional curvature of $(M,\o)$ and on $||v||_{C^1(M)}$.
We claim no originality in the proof below.

\begin{proof}[Proof of Theorem \ref{PRegThm} (i)]
\def\be{\beta}\def\al{\alpha}

It suffices, by a standard approximation procedure, to
assume that $v$ is smooth. 
For simplicity of notation, we will often denote 
$u_\beta$ by just $u$. 
The argument follows \cite[Theorem 1]{BlockiGrad} very closely. 
Let $B'$ be some sufficiently large positive constant to be fixed later.
Let $C_0:=\sup_{\beta>2}||u_\be||_{C^0}+1$.
Let $\phi: M \to \Bbb R$ be the following function:
$$\phi:= \log|\del u|_\o^2 - \gamma (u),$$
where 
$\gamma:[-C_0,C_0]\to \Bbb R$ is a 
smooth non-decreasing
function to be fixed later.
Since $\| u\|_{C^0} \leq C(M,\| v\|_{C^0}),$ independently of $\beta$ 
\cite{BeApprox}, $\gamma$ is thus defined on some fixed finite interval.

Suppose $\phi$ attains its maximum at $p \in M$. Let $z=(z_1,\ldots,z_n)$ denote 
holomorphic normal coordinates  around this point. 
Let $g$ denote a local potential for $\o$ in this chart, i.e., $\i\del\dbar g =\o$. Set $h := g + u$. 
We can additionally suppose that $\i\del\dbar u(p)$ is diagonal in our coordinates. 
Since all our local calculations will be carried out  at the point $p$ we omit the dependence on this point 
from the subsequent computations. 
Let 
$$
\alpha:=|\del u|_\o^2.
$$
Thus,
\begin{equation}\label{first_order_id}
0 = \frac{\partial}{\partial z_j}\phi = \phi_j = \frac{\alpha_j}{\alpha} - \gamma'(u) u_j, \q j = 1,\ldots,n,
\end{equation}
and so (omitting from now and on symbols for summation that can be understood from the context),
\begin{flalign}
0 
\geq 
\Delta_{\omega_u} \phi 
= 
\frac{\phi_{k\bar k}}{h_{k\bar k}} 
= 
\frac{1}{h_{k\bar k}}\Big(\frac{\alpha_{k\bar k}}{\alpha} - \frac{|\alpha_k|^2}{\alpha^2} 
- 
\gamma' u_{k\bar k} - \gamma''\al\Big)
=
\frac{1}{h_{k\bar k}}\Big(\frac{\alpha_{k\bar k}}{\alpha} - \gamma' u_{k\bar k} - (\gamma''+\gamma'^2) \al\Big),
\end{flalign}
The next formula holds for each fixed $k=1,\ldots,n$ (no summation)
\begin{flalign}
\alpha_{k\bar k} = 2 \Re \ u_{jk\bar k} u_{\bar j} + |u_{jk}|^2 + |u_{j\bar k}|^2 - u_j g_{j\bar l k \bar k}u_{\bar l}
\geq 
2 \Re \ u_{jk\bar k} u_{\bar j} + |u_{jk}|^2  - B\al,
\end{flalign}
whenever $-B$ is a lower bound for the bisectional curvature of $\o$.
Using this, the identity $1 + u_{k\bar k}=h_{k\bar k}$, and fact that $g_{jk\bar k}=0$,
and summing over $k$ we have, 
\begin{flalign*}
0 \geq &\frac{1}{h_{k\bar k}}\Big(\frac{2 \Re \ h_{jk\bar k} u_{\bar j} + |u_{jk}|^2- B\al}{\alpha} + \gamma' -\gamma'h_{k\bar k} - (\gamma''+\gamma'^2) \al \Big),
\end{flalign*}
Multiplying across with $\alpha$,  
\begin{flalign}
0 \geq 
\frac{1}{h_{k\bar k}}\Big(2 \Re \ h_{jk\bar k} u_{\bar j} + |u_{jk}|^2 
+ 
\al\big[\gamma'-B -\gamma'h_{k\bar k} - (\gamma''+\gamma'^2) \al\big] \Big) 
\label{int_est}
\end{flalign}
By \Blocki's trick \cite[(1.15)]{BlockiGrad}, we also have the following estimate:
\begin{equation}\label{blocki_trick}
\frac{|u_{jk}|^2}{h_{k\bar k}} \geq \al \Big( \gamma'^2 \frac{|u_k|^2}{h_{k\bar k}} -2\gamma' \Big) -2.
\end{equation}

The computations so far are general and taken from \cite{BlockiGrad}. We now bring
the equation we are interested in,
$\log\frac{\det[h_{j\bar l}]}{\det[g_{j\bar l}]} = \beta (u - v)$,
into the picture.
Differentiating this equation at $p$ yields
$\frac{h_{jk\bar k}}{h_{k\bar k}}= \beta (u_j - v_j).$
Thus, 
\begin{equation}\label{3deriv_est}
2\Re \ \frac{h_{jk\bar k}}{h_{k\bar k}}u_{\bar j}= 2\Re \ \beta (u_j - v_j)u_{\bar j} = 2\beta |u_j|^2  - 2\beta \Re \ v_j u_{\bar j}
\geq 2\beta |u_j|^2 - 2\beta|v_j|^2.
\end{equation}
Putting \eqref{3deriv_est} and \eqref{blocki_trick} into \eqref{int_est} we obtain:
\begin{flalign}
0 \geq& 2\beta \al - 2\beta |v_j|^2 + \al \Big( \gamma'^2 \frac{|u_k|^2}{h_{k\bar k}} -2\gamma' \Big) 
-2 + \frac{\al}{h_{k\bar k}}\Big[\gamma' -B-\gamma'h_{k\bar k} - (\gamma''+\gamma'^2) \al\Big] \nonumber \\
=& 2(\beta-\gamma') \al - 2\beta |v_j|^2  -2 -n\gamma'+ \frac{\al}{h_{k\bar k}}(\gamma' -B - \gamma'' \al)
\end{flalign}
Our wish is to get rid of the last term in the right. 
For this reason, we choose $\gamma:[-C_0,C_0] \to \Bbb R$ to be $\gamma(t)=-t^2/2 + (C_0 + B)t$. Then $2C_0 + B > \gamma'>B$, $\gamma'' <0$. With this choice, in our last estimate the rightmost term becomes positive, so we can write:
\begin{flalign}
0 \geq (2\beta-2C_0 - B) \al - 2\beta |v_j|^2  -2 -n(2C_0+B).
\end{flalign}
This gives
\begin{flalign}
\al \leq \frac{ 2\beta |v_j|^2   -2 -n(2C_0+B)}{2\beta-2C_0 - B},
\end{flalign}
concluding the proof of Theorem \ref{PRegThm} (i), since the constant on the right hand side can
be majorized independently of $\beta$.
\end{proof}
We turn to prove a corollary of this estimate and the formula for
the Bremermann upper envelope 
$\vp$ introduced in \eqref{uUpperEnvEq}
(Theorem \ref{envdual}), namely,  the
Lipschitz regularity of $\vp$.

\begin{proof}[Proof of Corollary \ref{LipschitzHCMACor}] It follows from the definition of $\vp$ that $\| \vp\|_{C^0} \leq \| v \|_{C^0}$. To finish the proof we need to prove that 
\begin{equation}\label{GradEst}
|\vp(s,\cdot)|_{C^{0,1}} \leq C(M,\o,\sup_{s \in \partial K}\|v(s,\cdot)\|_{C^1}), \ s \in \partial K.
\end{equation}
Fix $h \in K$. By \eqref{main_identity} we have
\begin{equation}
\vp(h,z)=(P\{v_s-\langle s,\sigma\rangle\}_{s\in\del K})^\star(h,z)
=\sup_{\sigma\in\RR^k}[P\{v_s-\langle s-h,\sigma\rangle\}_{s\in\del K}(z)], \ z \in M.
\end{equation}
Fix $\sigma \in \Bbb R^k$. As $K$ is bounded, 
by Lemma \ref{LipFamilyLemma} below, $\phi_\sigma := \inf_{s \in \del K}(v_s-\langle s-h,\sigma\rangle) \in C^{0,1}(X)$, 
with $|\phi_\sigma |_{C^{0,1}} 
\leq C(\sup_{s \in \partial K}|v(s,\cdot)|_{C^{0,1}})$. 
By Theorem \ref{PRegThm} (i) it follows that
$$| P(\phi_\sigma)|_{C^{0,1}} \leq C(|\phi_\sigma|_{C^{0,1}}) \leq  C(\sup_{s \in \partial K}|v(s,\cdot)|_{C^{0,1}}).$$
As $\vp(h,\cdot)=\sup_{\sigma \in \Bbb R^k}P(\phi_\sigma)$, \eqref{GradEst} follows from another application of Lemma
\ref{LipFamilyLemma}.
\end{proof}
The next lemma is a consequence of the Arzel\`a-Ascoli compactness theorem.

\begin{lemma} \label{LipFamilyLemma}Suppose $\{ f_\alpha\}_{\alpha \in A} \subset C^{0,1}(M)$ with $\sup_{\alpha \in A} |f_\alpha|_{C^{0,1}} < \infty$. Then:

\noindent
(i) Either $\phi:=\inf_{\alpha \in A} f_\alpha \equiv -\infty$, or $\phi \in C^{0,1}(M)$ with $|\phi|_{C^{0,1}} \leq \sup_{\alpha \in A} |f_\alpha|_{C^{0,1}}$. 

\noindent(ii) Either $\psi:=\sup_{\alpha \in A} f_\alpha \equiv \infty$, or $\psi \in C^{0,1}(M)$ with $|\psi|_{C^{0,1}} \leq \sup_{\alpha \in A} |f_\alpha|_{C^{0,1}}$.
\end{lemma}

\subsection{Regularity of rooftop convex- and psh-envelopes }
\label{RegRoofCVXPSHSubSec}

In this subsection we prove Theorem \ref{PRegThm} by using
Theorem \ref{P_reg_prop}. The proof of the latter is postponed
to \S\ref{proofofP_reg_prop}. First, we recall the second order estimates
of Berman \cite[Theorem 1.1, Remark 1.8]{Be}
and Berman--Demailly \cite[Theorem 1.4]{BD}:

\begin{theorem}
\label{BDreg} Let $b \in C^{1,1}(M)$. Then,
(i)
$\|P(b)\|_{C^{1\bar1}} \leq C(\|b\|_{C^{1\bar1}})$, and
(ii) If $[\omega_0] \in H^2(M,\Bbb Z)$, then $\|P(b)\|_{C^{2}} \leq C(\| b\|_{C^{2}}).$
\end{theorem}

\begin{proof}[Proof of Theorem \ref{PRegThm}]

For both parts (i) and (ii) we first assume
$v_0,v_1 \in C^{1,1}(M)$.
Indeed, by an approximation argument, this suffices also for treating part (i).

Take a covering of $M$ by charts, that we assume without loss of generality are unit balls of the form $\{B_1(x_j)\}_{j = 1}^k$ (possible as $M$
is compact), such that the balls $\{B_{1/8}(x_j)\}_{j = 1}^k$ still cover $M$.
Let $\{\rho_j\}_{j=1}^k$ be a partition of unity subordinate to the
latter covering.
Without loss of generality,
we also
assume that in a neighborhood of each
$B_1(x_j)$ the metric
$\omega$  has a \K potential $w_j \in C^\infty$.

Let $h_j \in \SH({B_1(x_j)})$ be the upper envelope
$$
h_j:=
\sup \{v \in \SH(B_1(x_j))\,:\,  v \leq
\min\{v_0|_{B_1(x_j)}+w_j,v_1|_{B_1(x_j)}+w_j\}\}.
$$
If $\vp$ is an $\o$-psh function then
$w_j+\vp\in\PSH(B_1(x_j))$ and therefore
$w_j+\vp\in\SH(B_1(x_j))$. Thus,
$P(v_0,v_1)|_{B_1(x_j)} \leq h_j-w_j \leq \min \{ v_0,v_1\}|_{B_1(x_j)}$, and
by Theorem \ref{P_reg_prop},
\beq\label{hjEq}
\| h_j\|_{C^{2}}
\leq
C(\| w_j\|_{C^{2}},\| v_0\|_{C^{2}},\| v_1\|_{C^{2}}).
\eeq
Set
$b:= \sum_{j=1}^k \rho_j (h_j - w_j)$.
Then
\beq\label{bC2EstEq}
\| b\|_{C^{2}}
\leq
C(\{\| w_j\|_{C^{2}}\}_{j=1}^k,\{\| \rho_j\|_{C^{2}}\}_{j=1}^k,\| v_0\|_{C^{2}},\| v_1\|_{C^{2}})\le C(M,\o,\| v_0\|_{C^{2}},\| v_1\|_{C^{2}}).
\eeq
It follows that $P(v_0,v_1) \leq b \leq \min\{ v_0,v_1\}$ as we noticed above that $P(v_0,v_1)|_{B_1(x_j)} \leq h_j-w_j \leq \min\{ v_0,v_1\}|_{B_1(x_j)}$.
Thus, $P(b)=P(v_0,v_1)$
and so part (i) of the theorem follows from \eqref{bC2EstEq}
and Theorem \ref{BDreg}.
Part (ii) follows as well if we can show that
$$\| b\|_{C^{1\b1}}
\leq
C(\{\| w_j\|_{C^{1\b1}}\}_{j=1}^k,\{\| \rho_j\|_{C^{1\b1}}\}_{j=1}^k,\| v_0\|_{C^{1\b1}},\| v_1\|_{C^{1\b1}})\le C(M,\o,\| v_0\|_{C^{1\b1}},\| v_1\|_{C^{1\b1}}).
$$

This estimate indeed holds since on the incidence set
$\{h_j-w_j=\min\{v_0|_{B_1(x_j)},v_1|_{B_1(x_j)}\}$ the function
$\Delta h_j$ equals either $\Delta v_0|_{B_1(x_j)}$
or $\Delta v_1|_{B_1(x_j)}$ a.e. with respect to the Lebesgue measure, while $h_j$ is harmonic on the complement of the incidence set.
\end{proof}

Corollary \ref{PRealRegThm} follows from the previous theorem because
by Proposition \ref{HRMAProp} the convex rooftop envelope solves the HRMA,
and hence also the HCMA on the associated toric manifold. 

\begin{remark}
{\rm
One can give a different proof of part (ii) of Theorem \ref{PRegThm}
using results on regularity of Mabuchi geodesics
together with Theorem \ref{envdual} and Proposition \ref{family_est}.
Indeed, let $[0,1] \ni t \to a_t \in \PSH(M,\omega) \cap L^\infty(M)$ be the weak geodesic joining $a_0 :=P(v_0)$ with $a_1 :=P(v_1)$.
By Theorem \ref{BDreg} both $P(v_0)$ and $P(v_1)$ have bounded Laplacian. By
Berman--Demailly \cite[Corollary 4.7]{BD} (see He \cite{He} for a different proof)
 so does each $a_t$ for each $t \in [0,1]$.
Since $P(v_0,v_1)=P(P(v_0),P(v_1))$, by Theorem
\ref{envdual} we have
$$P(v_0,v_1)=a^*_0.$$
Finally,
$| \Delta_{\omega_0} a^*_0|$ is bounded
by Proposition \ref{family_est} below.
}
\end{remark}
The following estimate is very likely well-known, 
although we were not able to find a precise reference.
\begin{proposition}\label{family_est}
Let $\{ v_a\}_{a \in A}$ be a
uniformly locally bounded
 family of functions on a domain $D \subset \Bbb C^n$. Suppose that
 $| \Delta v_a|\leq B$
for all $a \in A$,
and that $v_{\min}=\inf _{a \in A}v_a$  is psh on $D$. Then,
$| \Delta v_{\min}| \leq B$.
\end{proposition}

One can also assume instead of uniform local boundedness that $v_{\min}$
itself is locally bounded.
\begin{proof} By our assumption $\Delta v_{\min} \geq 0$, hence we only have to prove that $\Delta v_{\min} \leq B$. Our assumptions also imply that the functions $B|z|^2/2n - u_a$
are subharmonic on $D$ for any $a \in \mathcal A$. By the Zygmund-Calderon estimate we also have that the $C^{0,1}$ norm of the functions $B|z|^2/2n - u_a$ is uniformly bounded on any relatively compact open subset of $D$. This implies that $B|z|^2/2n - v_{\min}=\sup_{a \in \mathcal A}(B|z|^2/2n - v_a)$ is locally Lipschitz continuous, hence by Choquet's lemma also subharmonic. This in turn implies that $\Delta v_{\min} \leq B$.
\end{proof}

\subsection{Regularity of rooftop subharmonic-envelopes}
\label{proofofP_reg_prop}

We now prove Theorem \ref{P_reg_prop}. Let us fix some notation. Let
$b_0,b_1 \in
C^{1,1}({B_1})$ with $B_1\subset \Bbb C^n =\Bbb R^{2n}$.
The envelope $\bmin$ \eqref{bminEq} is  upper semi-continuous hence it is subharmonic by Choquet's lemma.
We call $\benv$ the {\it subharmonic-envelope of the rooftop obstacle
$\min \{ b_0,b_1\}$}.

Let
\beq\label{b10Eq}
b_{10}:=
b_1 - b_0\in C^{1,1}({B_1}),
\eeq
and denote the {\it contact set} (or {\it coincidence set}) by
\beq\label{LambdaEq}
\Lambda:= \{x\in B_1\,:\, b_{\env}(x) = \min\{b_0,b_1 \}(x)\}. 
\eeq
We call the complement of $\Lambda$ in $B_1$ the {\it non-coincidence set}.

Our first result assures that whenever $x_0$ is a regular point of the level set $b_{10}^{-1}(0)$,
then $x_0$ is contained in the non-coincidence set,
along with a small open ball of radius uniformly
proportional to $|\nabla b_{10}(x_0)|$.

\begin{proposition}
\label{cushintheorem}
For $b_0,b_1 \in C^{1,1}({B_1})$,
using the notation we introduced, there exists
$C=C(n)/(1 + \|b_0\|_{C^{2}}+\|b_1\|_{C^{2}})$
such that for any $x_0 \in b_{10}^{-1}(0)\cap B_{1/2}$
(recall (\ref{b10Eq}) and (\ref{LambdaEq})),
$$
\Lambda \cap B_{C|\nabla b_{10}(x_0)|}(x_0)  = \emptyset.
$$
\end{proposition}
\begin{proof} We fix $x_0 \in b_{10}^{-1}(0)\cap B_{1/2}$.
We will prove that $\benv<\min\{b_0,b_1\}$ on $B_{C|\nabla b_{10}(x_0)|}(x_0)$
by finding a linear function sandwiched between these two functions.
More precisely, the proposition follows from the estimate
\begin{equation}\label{cushin}
\benv(x) < b_0(x_0)-2C |\nabla b_{10}(x_0)|^2 + \langle \nabla b_0(x_0),x-x_0 \rangle
< \min\{b_0,b_1\}(x) , \ x  \in B_{C|\nabla b_{10}(x_0)|}(x_0),
\end{equation}
for $C$ as in the statement.

For the second inequality in \eqref{cushin},
observe that for any $x \in B_r(x_0), \ r \leq 1/2$,
\begin{flalign*}
\min\{b_0,b_1\}(x)-&b_0(x_0)
\geq
\min_{i\in\{0,1\}} \langle\nabla b_i(x_0),x-x_0 \rangle
-(\|b_0\|_{C^{2}}+\|b_1\|_{C^{2}})|x-x_0|^2 &\\
\geq&
\langle \nabla b_0(x_0),x-x_0 \rangle
+ \min\{0,\langle \nabla b_{10}(x_0),x-x_0 \rangle\}
- (\|b_0\|_{C^{2}}+\|b_1\|_{C^{2}})|x-x_0|^2  &\\
>&
\langle \nabla b_0(x_0),x-x_0 \rangle - |\nabla b_{10}(x_0)|r
- (\|b_0\|_{C^{2}}+\|b_1\|_{C^{2}})r^2.
\end{flalign*}
Set $r=r'|\nabla b_{10}(x_0)|$. Then,
whenever $r' \leq 1/(1 + \|b_0\|_{C^{2}}+\|b_1\|_{C^{2}})$,
$$
(\|b_0\|_{C^{2}}+\|b_1\|_{C^{2}})r^2
=
(\|b_0\|_{C^{2}}+\|b_1\|_{C^{2}})r'|\nabla b_{10}(x_0)|r
\le r' |\nabla b_{10}(x_0)|^2.
$$
Thus, as desired,
\begin{equation}\label{cushin1}
\min\{b_0,b_1\}(x) > b_0(x_0)-2r' |\nabla b_{10}(x_0)|^2
+ \langle \nabla b_0(x_0),x-x_0\rangle, \ x \in B_{r'|\nabla b_{10}(x_0)|}(x_0).
\end{equation}

Now we turn to the first inequality in \eqref{cushin}.
Fix $r \leq 1/2$.
As before, by Taylor's formula, for $x \in B_r(x_0)$,
\beq\label{minEst1Eq}
\min\{b_0,b_1\}(x)\leq
 b_0(x_0)+ \langle \nabla b_0(x_0),x-x_0 \rangle
+\min\{0,\langle \nabla b_{10}(x_0),x-x_0 \rangle\} +
(\|b_0\|_{C^{2}}+\|b_1\|_{C^{2}})r^2.
\eeq
Note that $h$ is subharmonic,
while
$B_r(x_0)\ni x
\mapsto
\langle \nabla b_0(x_0),x-x_0 \rangle
+ (\|b_0\|_{C^{2}}+\|b_1\|_{C^{2}})r^2$ is harmonic.
Combining this with \eqref{minEst1Eq} and the fact that
$h \leq \min\{ b_0,b_1\}$, it follows that
\begin{flalign*}
\benv(x) \leq& b_0(x_0)+ \langle \nabla b_0(x_0),x-x_0 \rangle
+
\int_{\partial B_r(x_0)} P_{ r}(x-x_0,\xi)\min\{0,\langle\nabla b_{10}(x_0),\xi\rangle\}d\sigma(\xi)
\\&+ (\|b_0\|_{C^{2}}+\|b_1\|_{C^{2}})r^2,
\end{flalign*}
where
$P_{r}(x,\xi)=(r^2 - |x|^2)/(2n\omega_{2n} r|x-\xi|^{2n})$
 is the Poisson kernel
of the ball $B_r(x)$ which is positive.
For any $x \in B_{r/2}(x_0)$ and $\xi\in \partial B_r(x_0)$,
there is a uniform estimate $|P_{ r}(x-x_0,\xi)|\le C(n)r^{1-2n}$.
Also, $\langle\nabla b_{10}(x_0),\xi\rangle = |\nabla b_{10}(x_0)||\xi|\cos\alpha,$
where $\alpha$ is the angle betwen $\xi$
and $\nabla b_{10}(x_0)$ in the plane they generate.
Now,
since the integrand is negative, one can estimate
it by considering only the quarter sphere
$\partial B^{++}_r(x_0)$ where the angle between $\xi$ and $x-x_0$
is in the range $(-\pi/4,\pi/4)$. Then,
$$
\int_{\partial B_r(x_0)} P_{ r}(x-x_0,\xi)\min\{0,\nabla b_{10}(x_0)\xi\}d\sigma(\xi)
<
\frac1{\sqrt2}
\int_{\partial B^{++}_r(x_0)} P_{ r}(x-x_0,\xi)
|\nabla b_{10}(x_0)|rd\sigma(\xi),
$$
which, in turn, is bounded from above by $-C|\nabla b_{10}(x_0)|r$ for $x \in B(x_0, r/2)$.
Thus,
there exists $C'=C'(n) < 1$ such that
$$
\benv(x)\leq  b_0(x_0)+\langle \nabla b_0(x_0),x-x_0 \rangle - C'|\nabla b_{10}(x_0)|r+(\|
b_0\|_{C^{2}}+\|b_1\|_{C^{2}})r^2,
$$
$x \in B(x_0, r/2)$. By taking any
$
\tilde r\leq \frac{C'}{2(1 + \|b_0\|_{C^{2}}+\|b_1\|_{C^{2}})}$
one has that
$\tilde r |\nabla b_{10}(x_0)| < 1$. Thus,
$$
\benv(x)
\leq
b_0(x_0)+
\langle \nabla b_0(x_0),x-x_0 \rangle
-\frac{C'}{2}|\nabla b_{10}(x_0)|^2\tilde r,
\quad
\h{for any $x\in B_{\tilde r|\nabla b_{10}(x_0)|/2}(x_0)$}.
$$
Therefore, for any  choice $r'' < {C'\tilde r}/{4}$,
\begin{equation}\label{cushin2}
\benv(x)
<
b_0(x_0)+\langle \nabla b_0(x_0),x-x_0 \rangle - 2|\nabla b_{10}(x_0)|^2r'',
\quad
\h{for any $x \in B_{r''|\nabla v(x_0)|}(x_0)$}.
\end{equation}
The estimate \eqref{cushin} with
$C=\min \{r',r''\}$ follows from \eqref{cushin1} and \eqref{cushin2}.
\end{proof}

Before we consider the interior regularity of $\benv$,
we prove an adaptation to our setting
of the standard quadratic growth lemma (cf. \cite[Lemma 3]{Caf}).
It shows, roughly, that the envelope $\benv$ approximates the obstacle
$\min\{b_0,b_1\}$ at least to second order.
This is quite intuitive in the classical case of an obstacle of class
$C^{1,1}$. In our setting where the obstacle is only Lipschitz,
the proof relies on Proposition \ref{cushintheorem}.

\begin{proposition}
\label{QuadraticProp}
Let $b_0,b_1 \in C^{1,1}({B_1}(x_0))$.
Suppose $x_0\in\Lambda\cap B_{1/4}(x_0)$,
with
$\min \{ b_0,b_1\}(x_0)= b_i(x_0)$ for $i\in\{0,1\}$.
Then, there exists $C = C(\| b_0\|_{C^{2}(B_1)},\| b_1\|_{C^{2}(B_1)})$
such that (recall
(\ref{LambdaEq}) and (\ref{bminEq})),
\beq\label{QuadGrowthEstEq}
|\benv(x) - b_i(x_0) - \langle \nabla b_i(x_0), x - x_0\rangle| \leq C |x-x_0|^2,
\quad
\h{for all $x \in
B_{1/8}(x_0)$}.
\eeq
\end{proposition}

Of course, $\benv$ equals $b_i$ up to infinite order on the interior of
$\Lambda$, so
one could phrase \eqref{QuadGrowthEstEq} as
$$
|\benv(x) - \benv(x_0) - \langle \nabla \benv(x_0), x - x_0\rangle| \leq C |x-x_0|^2,
$$
whenever $x_0$ lies in the interior of $\Lambda$.
However, the key is, of course, that the estimate \eqref{QuadGrowthEstEq}
also holds on $\partial\Lambda$ (the {\it free boundary}),
and it precisely shows that $\benv$ is therefore differentiable
at points on $\partial\Lambda$, and in fact that its $C^{1,1}$ norm there
is uniformly bounded. These are the problematic points,
since $\benv$ is harmonic (and thus, well-behaved) on the complement
of $\Lambda$.

\begin{proof} 
Let $x \in \Lambda\cap B_{1/4}(x_0)$,
and suppose that
$\min\{b_0,b_1\}(x_0)=b_0(x_0)$ (the case $i=1$ is treated in the same manner).
Set
\beq\label{MEq}
M := \| b_0\|_{C^{2}}.
\eeq
 Then,
\beq\label{benvaboveEstEq}
\benv(x) - b_0(x_0) -
\langle \nabla b_0(x_0), x - x_0\rangle \leq \benv(x) - b_0(x)+ M|x -x_0|^2  \leq  M|x-x_0|^2.
\eeq
Hence, it remains
to prove that
\begin{equation}\label{quadest}
-C|x-x_0|^2 \leq \benv(x) - b_0(x_0) - \langle \nabla b_0(x_0), x - x_0\rangle,
\quad
\h{for all $x \in
B_{1/8}(x_0)$
}.
\end{equation}
Fix now $r \leq 1/4$. On $B_r(x_0)$,
decompose
\beq\label{sEq}
s(x):=\benv(x) - b_0(x_0) - \langle \nabla b_0(x_0), x - x_0\rangle - Mr^2
\eeq
into the sum
$s|_{B_r(x_0)}= s_1 + s_2$, with
$s_1$ is harmonic on $B_r(x_0)$ with $s_1|_{\partial B_r(x_0)}=s|_{\partial B_r(x_0)}$.

Since $s_1$ is harmonic, $s \leq s_1 \leq0$.
Also, by the Harnack inequality for non-positive harmonic functions
it follows that
\begin{equation} \label{s_1}
-Mr^2 =s(x_0)\leq s_1(x_0) \leq C \inf_{B_{r/2}(x_0)} s_1,
\end{equation}
with $C$ independent of $r$.

\begin{claim}
\label{muClaim}
Let $\mu_{s_2}$ denote the measure associated to $\Delta s_2$.
Then either $s_2 \equiv 0$ or $\inf_{x \in B_r(x_0)} s_2$ is attained inside $B_r(x_0)$ on the support of $\mu_{s_2}$.
\end{claim}

\bpf
First, since the obstacle $\min\{b_0,b_1\}$ is Lipschitz, it follows
from \cite[Lemma 3(a)]{Caf} that $\benv$ is Lipschitz.
In particular, $\benv$ is continuous and so $\inf_{x \in \overline{B_r(x_0)}} s_2$ is attained.

Now, suppose that the infimum is attained at a point $p$ on the complement of the support of $\mu_{s_2}$.
By definition of support, there is an open ball containing $q$ on which $s_2$ is harmonic. But, a harmonic function cannot obtain an interior minimum, which implies
that $p$ must be on the boundary of $B_r(x_0)$. But we have $s_2|_{\partial B_r(x_0)}=0$ and $s \leq s_1 \leq 0$ implies $s_2 \leq 0$. Hence, if the infimum of $s_2$ is obtained on the boundary then $s_2 \equiv 0.$
\epf
If $s_2 \equiv 0$ then \eqref{quadest} follows from \eqref{s_1}. Hence we can suppose that $\inf_{x \in B_r(x_0)}s_2$ is attained at $x_1 \in  B_r(x_0)$. By Claim \ref{muClaim}, $x_1\in\Lambda$ since the support of $\mu_{s_2}$
in  $B_r(x_0)$ is
equal to the support $\mu_{\benv}$ (the measure associated to $\Delta\benv$) in  $B_r(x_0)$ that is, in turn, contained in
$\Lambda\cap B_r(x_0)$.
Suppose first that $\min\{b_0,b_1\}(x_1)=b_0(x_1).$
Thus, since $x_1 \in \Lambda$, $\benv(x_1)=b_0(x_1)$. Thus, using \eqref{MEq}
and \eqref{sEq},
\begin{equation} \label{s_2a}
\inf_{B_r(x_0)} s_2
=
s_2(x_1)
\ge s(x_1)= b_0(x_1) - b_0(x_0) - \langle \nabla b_0(x_0), x_1 - x_0\rangle - Mr^2  \geq -2Mr^2.
\end{equation}
Combining \eqref{benvaboveEstEq}, \eqref{s_1}, and \eqref{s_2a}
and the definition of $s$ \eqref{sEq},
proves \eqref{quadest} in this case.

\begin{figure}
\begin{center}
\includegraphics[trim=3cm 21cm 3cm 0cm, clip=true,width=3.9in]{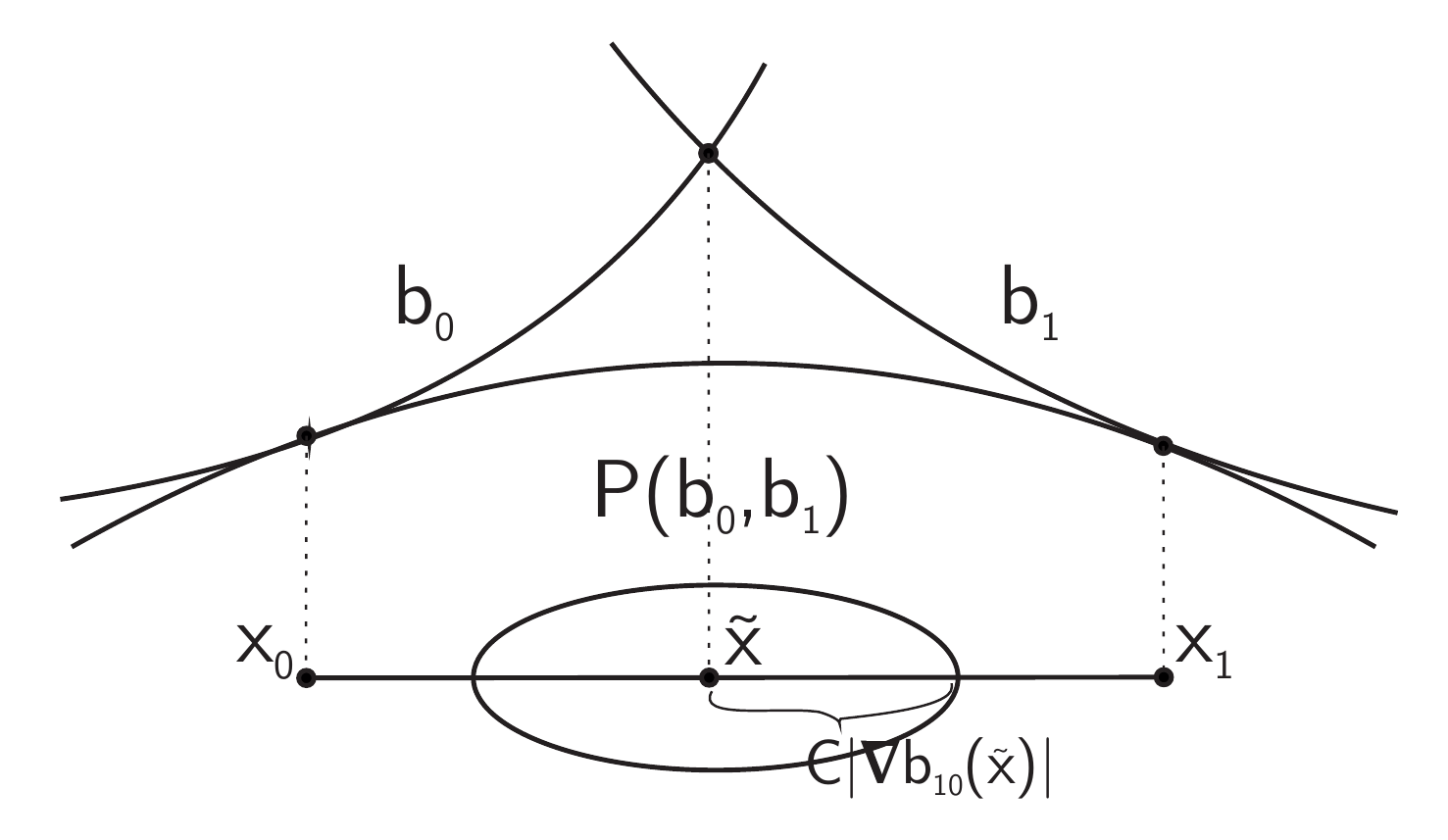}
\caption{The barriers $b_0,b_1$ and the envelope $P(b_0,b_1)$}
\label{figP(b0b1)}
\end{center}
\end{figure}

Suppose now that $\min\{b_0,b_1\}(x_1)=b_1(x_1)$
(see Figure \ref{figP(b0b1)}).
This case is new compared with the classical setting of Caffarelli \cite{Caf}
and will rely crucially on Proposition \ref{cushintheorem}.
Since $\min\{b_0,b_1\}(x_0)=b_0(x_0)$, it follows by continuity of $b_0$ and $b_1$
that there exists a point $\tilde x$
on the straight line segment $\{(1-t)x_0+tx_1\,:\, t\in[0,1]\}$ connecting $x_0$ and $x_1$
 such that $b_1(\tilde x)=b_0(\tilde x)$, i.e. $\tilde x \in
b_{10}^{-1}(0)\cap B_{r}(x_0)$. Hence,
$$
\baeq
\label{s2_est}
\inf_{B_r(x_0)} s_2
=
s_2(x_1)
\ge s(x_1)
& =  b_1(x_1) - b_0(x_0) - \langle \nabla b_0(x_0), x_1 - x_0\rangle -Mr^2
\cr
& =
(b_1(x_1) - b_1(\tilde x) -
\langle \nabla b_1(\tilde x), x_1 - \tilde x\rangle) &\nonumber
\cr
&\q+ (\langle \nabla b_1(\tilde x), x_1 - \tilde x\rangle - \langle \nabla b_0(\tilde x), x_1 - \tilde x\rangle) & \nonumber
\cr
&\q+ (\langle \nabla b_0(\tilde x), x_1 - \tilde x\rangle - \langle \nabla b_0(x_0), x_1 - \tilde x\rangle)
&
\cr
&\q + (b_0(\tilde x) - b_0(x_0) - \langle \nabla b_0(x_0), \tilde x - x_0\rangle) - Mr^2\nonumber
\eaeq
$$
We now estimate from below the last four lines.
The first line is minorized by $-2\|b_1
\|_{C^{2}}|x_1-\tilde x|^2\ge -cr^2$, while
the third and fourth lines are
minorized by $-2\|b_0\|_{C^{2}}(|x_1-\tilde x|^2+|x_0-\tilde x|^2+r^2)\ge-cr^2$
(recall \eqref{MEq} and that $|x_1-x_0|\le r$, thus $|x_i-\tilde x|\le r$), for some $c=c(||b_0||_{C^{2}},||b_1||_{C^{2}})$.
In sum,
\beq\label{InSumEq}
\inf_{x \in B_r(x_0)} s_2
\ge
(\langle \nabla b_1(\tilde x), x_1 - \tilde x\rangle - \langle \nabla b_0(\tilde x),
x_1 - \tilde x\rangle)- Cr^2
\geq -| \nabla b_{10}(\tilde x)| |x_1 - x_0|- Cr^2.
\eeq
Now, by Proposition \ref{cushintheorem},
for some
$C=C(n)/(1 + \|b_0\|_{C^{2}}+\|b_1\|_{C^{2}})$,
there is a ball of radius $C|\nabla b_{10}(\tilde x)|$
around $\tilde x$ that does not intersect $\Lambda$.
But $x_0,x_1$ are both in $\Lambda$. Thus,
$$
C|\nabla b_{10}(\tilde x)| \le |x_i - \tilde x|, \q \h{for $i=0,1$},
$$
hence,
$$
2C|\nabla b_{10}(\tilde x)| \le |x_1 - x_0|.
$$
Plugging this back into \eqref{InSumEq} yields
\begin{equation}\label{s_2b}
\inf_{B_{r/2}(x_0)} s_2 \ge
\inf_{B_r(x_0)} s_2\geq - C'r^2,
\end{equation}
for $C'=C'(||b_0||_{C^{2}},||b_1||_{C^{2}})$.
Thus, \eqref{quadest} holds also in this case.
This concludes the proof of the Proposition.
\end{proof}

Finally, we are in a position to prove the interior $C^{1,1}$ regularity of $\benv$
 \eqref{bminEq}.

\begin{proposition}
\label{InteriorProp}
\label{localreg} Let $b_0,b_1 \in C^{1,1}({B_1})$.
There exists $C = C(\| b_0\|_{C^{2}},\| b_1\|_{C^{2}})$ such that
$$
\| \benv\|_{C^{2}(B_{1/8})}\leq C.
$$
\end{proposition}

\begin{proof}
First, $\benv$ is differentiable on $ B_{1/4}$.
This is immediate on $\Lambda^c\cap B_{1/4}$ since $\benv$ is harmonic there,
while
on $\Lambda\cap B_{1/4}$ this follows from Proposition \ref{QuadraticProp}.
Now, $\nabla h$ is Lipschitz continuous on $B_{1/8}$ with
Lipschitz constant $C$ if 
$$
|\benv(x) - \benv(x_0) - \langle \nabla \benv(x_0), x - x_0\rangle| \leq C |x-x_0|^2,
\quad \forall x_0,x \in B_{1/8}.
$$
This is shown in Proposition \ref{QuadraticProp} for $x_0 \in \Lambda \cap B_{1/8}$,
so suppose that
$x_0 \in \Lambda^c \cap B_{1/8}$.
Denote by $\rho$
the distance of $x_0$ to $\Lambda$.
If $\rho > 1/16$, then we are done since $\benv$ is harmonic on $B_{1/8}(x_0)$
and so $||\benv||_{C^{2}(B_{1/8}(x_0))}\le C||\benv||_{L^\infty(B_{1/4}(x_0))}
\le C(\|b_0\|_{L^\infty(B_{1/4}(x_0))},
\|b_1\|_{L^\infty(B_{1/4}(x_0))})$
(here we used the fact that (i) $\benv\le\min\{b_0,b_1\}\le\min\{\max b_0,\max b_1\}\}$,
(ii) since $b_0,b_1$ are bounded from below, the constant function $\min\{\min b_0,\min b_1\}$
is a candidate in the supremum for $\benv$; thus, $\benv\ge \min\{\min b_0,\min b_1\}$,
(iii) the $C^k$ norm of a harmonic function on a half-ball is estimated by its $C^0$
norm on the ball, divided by the radius of the ball to the $k$-th power---this
follows from the Poisson representation formula).
If $\rho\leq 1/16$ a different argument is needed since the radius of the ball
on which $\benv$ is harmonic can be arbitrarily small. Thus, let $x_1 \in \partial\Lambda \cap B(0,1/4)$
be a point at distance exactly $\rho$ from $x_0$.
Since $\benv$ is harmonic on $B_{\rho}(x_0)$
so is $\benv(x) - \benv(x_1) - \langle\nabla \benv(x_1),x-x_1\rangle$.
Thus, one may express the latter
in terms of its boundary values and the Poisson kernel.
Since,
$$
\nabla^2\benv
=
\nabla^2(\benv(x) - \benv(x_1) - \langle\nabla \benv(x_1),x-x_1\rangle)
$$
then differentiating the aforementioned integral representation twice under the integral
sign yields that
$$
\|\nabla^2 \benv(x_0)\|
\leq
C
\frac{\sup_{x \in B_{\rho}(x_0)}|\benv(x) - \benv(x_1) - \langle\nabla \benv(x_1),x-x_1\rangle|}{\rho^2}.
$$
Finally, since $B_{\rho}(x_0)\subset B_{2\rho}(x_1)$ it follows
from Proposition \ref{QuadraticProp} that the right hand side is majorized by
$$
C\frac{\sup_{x \in B(x_1,2\rho)}|\benv(x) - \benv(x_1) - \langle\nabla \benv(x_1),x-x_1\rangle|}
{\rho^2}
\leq C,
$$
as desired.
\end{proof}

\section*{Acknowledgments} 

The first version of this article was written in Spring 2013. The final touches took place a year later while YAR was visiting Chalmers
University, and he is grateful to the Mathematics Department for an excellent working atmosphere,
and to R. Berman and B. Berndtsson for making the visit possible and for their warm hospitality.
We thank them, as well as C. Kiselman, L. Lempert, and A. Petrosyan for 
their interest, encouragement, and related discussions.
This research was supported by NSF grants DMS-0802923, 1162070, 1206284,
and a Sloan Research Fellowship.

\begin{spacing}{0}

\def\bi{\bibitem}

\end{spacing}

\bigskip

\bigskip

{\sc Purdue University} 

{\tt tdarvas@math.purdue.edu}

\bigskip

{\sc University of Maryland} 

{\tt yanir@umd.edu}

\end{document}